\documentclass[12pt,oneside,reqno]{amsart}
\usepackage{amsthm,amscd,amssymb,verbatim,amsfonts,epsf,amsmath,mathrsfs,graphicx}
\usepackage{enumerate}
\usepackage[colorlinks=true,linkcolor=blue,citecolor=blue]{hyperref}
\def\line#1{\hbox to \hsize{#1\hfill}}
\newtheorem{prop}{Proposition}

\newtheorem{defi}{Definition}
\newtheorem{lemm}{Lemma}
\newtheorem{theo}{Theorem}

\newtheorem{coro}{Corollary}

\newcommand{\R}[1][]{\ensuremath{{\mathbb{R}^{#1}} }}

\title[Geometric structure on tangent bundles]%
{A new geometric structure on tangent bundles}

\begin{document} 

\author{Nikos Georgiou}
\address{Nikos Georgiou\\
  Department of Mathematics\\
          Waterford Institute of Technology\\
          Waterford\\
          Co. Waterford\\
          Ireland.}
\email{ngeorgiou@wit.ie }

\author{Brendan Guilfoyle}
\address{Brendan Guilfoyle\\
          School of STEM \\
          Institute of Technology, Tralee \\
          Clash \\
          Tralee  \\
          Co. Kerry \\
          Ireland.}
\email{brendan.guilfoyle@ittralee.ie}

\begin{abstract}
For a Riemannian manifold $(N,g)$, we construct a scalar flat metric $G$ in the tangent bundle $TN$. It is locally conformally flat if and only if either, $N$ is a 2-dimensional manifold or, $(N,g)$ is a real space form. It is also shown that $G$ is locally symmetric if and only if $g$ is locally symmetric. We then study submanifolds in $TN$ and, in particular, find the conditions for a curve to be geodesic. The conditions for a Lagrangian graph to be minimal or Hamiltonian minimal in the tangent bundle $T{\mathbb R}^n$ of the Euclidean real space ${\mathbb R}^n$ are studied. Finally, using the cross product in ${\mathbb R}^3$ we show that the space of oriented lines in ${\mathbb R}^3$ can be minimally isometrically embedded in $T{\mathbb R}^3$.
\end{abstract}

\date{14th June 2018}
\maketitle
\tableofcontents

\section{Introduction}

The geometry of the tangent bundle $TN$ of a Riemannian manifold $(N,g)$ has been a topic of great interest for the last 60 years. In the celebrated article \cite{Sasaki}, Sasaki used the Levi-Civita connection of $g$ to split the tangent bundle $TTN$ of $TN$ into a horizontal and a vertical part, constructing the first geometric structure  of $TN$. . Namely, one can obtain a splitting 
$TTN = HN \oplus VN$, where the subbundles $HN$ and $VN$ of $TTN$ are both isomorphic to the tangent bundle $TN$ - for more details see section \ref{s:intr}. For $\bar X\in TTN$, we write $\bar X\simeq (\Pi\bar X,K\bar X)$, where $\Pi\bar X\in HN$ and $K\bar X\in VN$.  Sasaki defined the following metric on $TN$ \cite{gudkappos}:
\[
(\bar X,\bar Y)\mapsto g(\Pi\bar X, \Pi\bar Y) + g(K\bar X, K\bar Y),
\]
The Sasaki metric is "rigid" in the following sense: it is scalar flat if and only if $g$ is flat \cite{Kowalski}. 

In the years since, several new geometries on the tangent bundle $TN$ have been constructed using the splitting of $TTN$ - see for example \cite{mustricerri} and \cite{yano}.   
When the base manifold $N$ admits additional structure one can use it to define other geometries in the tangent bundle. H. Anciaux and R. Pascal 
constructed in \cite{AnaR} a canonical pseudo-Riemannian metric in $TN$ derived from a K\"ahler structure on $N$. One example, the canonical neutral metric in  $T{\mathbb S}^2$ defined by the standard K\"ahler structure of the round 2-sphere ${\mathbb S}^2$, has been used to study classical differential geometry in ${\mathbb R}^3$ - see \cite{angui}, \cite{gk1} and \cite{gk2}.  

Using the Riemannian metric $g$ one can define a canonical symplectic structure $\Omega$ in $TN$. This can be achieved by using the musical isomorphism between the tangent bundle and the cotangent bundle. In this article we use the existence of an almost paracomplex structure $J$ on $TN$ compatible with $\Omega$ to construct a neutral metric $G$ on $TN$. 

If the base manifold $N$ admits a K\"ahler structure, the neutral metric $G$ and the pseudo-Riemannian metric, derived from the K\"ahler structure, are isometric - see Propostion \ref{p:isometricmetrics}. In other words, the neutral metric $G$ is a natural extension of the K\"ahler metric constructed by H. Anciaux and R. Pascal in \cite{AnaR} to the case where the base manifold does not admit a K\"ahler structure.

The purpose of this article is to study the geometric properties and submanifolds of the neutral metric $G$. We first prove the following:
\begin{theo}\label{t:oneofthema}
The neutral metric $G$ has the following properties:
\begin{enumerate}
\item $G$ is scalar flat,
\item $G$ is Einstein if and only if $g$ is Ricci flat,
\item $G$ is locally conformally flat if and only if either $n=2$ or, $g$ is of constant sectional curvature,
\item $G$ is locally symmetric if and only if $g$ is locally symmetric.
\end{enumerate}
\end{theo}
We then focus our attention on submanifolds. In particular, the geodesics of $G$ are characterized by:
\begin{theo}\label{t:geodesics}
A curve $\gamma(t)=(x(t),V(t))$ in $TN$ is a geodesic with respect to the metric $G$ if and only if the curve $x$ is a geodesic on $N$ and $V$ is a Jacobi field along $x$. 
\end{theo}

It was shown in \cite{AGR}, that the existence of a minimal Lagrangian graph in $TN$, where $(N,g)$ is a 2-dimensional Riemannian manifold, implies that $g$ is flat. A generalization of this result is given by the following Theorem:

\begin{theo}\label{t:ricciflat}
If $TN$ contains a Lagrangian graph with parallel mean curvature then the neutral metric $G$ is Ricci flat.
\end{theo}

To continue with submanifolds, we need to introduce some more terminology. A vector field $X$ in a symplectic manifold $(N,\omega)$, is called a {\it Hamiltonian field} if $\omega(X,.)=dh$, where $h$ is a smooth function on $N$. A Lagrangian immersion will be called {\it Hamiltonian minimal} if the variations of its volume along all Hamiltonian vector fields are zero. Let $\Sigma$ be a Lagrangian submanifold in a (para-) K\"ahler manifold $(N,g,\omega,j)$, and $H$ be its mean curvature. The first variation formula shows that $\Sigma$ is Hamiltonian minimal if and only if the tangential vector field $jH$ is divergence free \cite{Oh}.

For Lagrangian graphs in $T{\mathbb R}^n$ we prove the following:

\begin{theo}\label{t:oneofthemainthe}
Suppose that $u$ is a $C^4$-smooth function in an open set of ${\mathbb R}^{n}$ and let $f$ be the corresponding Lagrangian graph $f:{\mathbb R}^{n}\rightarrow T{\mathbb R}^{n}:p\mapsto (p,Du(p))$ in $T{\mathbb R}^{n}$. Then the following two statements hold true:
\begin{enumerate}
\item If $u$ is functionally related of second order then the induced metric $f^{\ast}G$ is flat.
\item The graph $f$ is Hamiltonian minimal if and only if $\log|\det Hess\, u|$, is an harmonic function with respect to the metric $G$ induced by $f$.
\item The graph $f$ is minimal if and only if $u$ satisfies the following Monge-Amp\'ere equation:
\begin{equation}\label{e:minimalcondition}
\det\mbox{Hess}(u)=c_0,
\end{equation}
where $c_0$ is a positive real constant.
Furthermore, $f$ is totally geodesic if and only if $u$ is of the following form:
\begin{equation}\label{e:totallycondition}
u(x_1,\ldots,x_n)=\sum_{1\leq i\leq j\leq n}(a_{ij}x_ix_j+b_ix_i+c).,
\end{equation}
where $a_{ij},b_i$ and $c$ are all real constants and $\det a_{ij}>0$. In particular, every totally geodesic Lagrangian graph is flat.
\end{enumerate}
\end{theo}

As an example of Theorem \ref{t:oneofthemainthe}, we focus our attention to the graph of a vector field in ${\mathbb R}^n$ of spherical symmetry, called source fields (see Definition \ref{d:source}). Such graphs are Lagrangian submanifolds in $T{\mathbb R}^n$. 

\begin{theo}\label{t:sourcefields}
Suppose that $f:{\mathbb R}^n-\{0\}\rightarrow T({\mathbb R}^n-\{0\}):p\mapsto (p,V(p))$ is the graph of a source field $V(p)=H(R)\,\partial/\partial R$, where $R=|p|$. Then the following two statements hold:
\begin{enumerate}
\item $f$ is minimal if and only if the field intensity $H$ is given by 
\[
H=(c_0R^n+c_1)^{1/n},
\]
where $c_0,c_1$ are constants with $c_0>0$. Furthermore, $f$ is totally geodesic if and only if $c_1=0$.
\item $f$ is Hamiltonian minimal (but non minimal) if and only if the field intensity $H$ is given by 
\[
H=\left(c_2+c_1e^{c_0R}\sum_{k=0}^{n-1}\frac{k!(-1)^k}{c_0^k}{n-1 \choose k}R^{n-k-1}\right)^{1/n},
\]
where $c_0\neq 0, c_1, c_2$ are real constants.
\end{enumerate}
\end{theo}



It is well known that the space ${\mathbb L}({\mathbb R}^3)$ of oriented lines in ${\mathbb R}^3$ is identified with the tangent bundle $T{\mathbb S}^2$. Guilfoyle and Klingenberg in \cite{gk1} and M. Salvai in \cite{Salvai}, studied the geometry of $({\mathbb L}({\mathbb R}^3),{\mathbb G},{\mathbb J})$ derived by the standard K\"ahler structure $(T{\mathbb S}^2,g,J)$. Using this identification we show the following:

\begin{theo}\label{t:minimalembed}
There exists a minimal isometric embedding of $({\mathbb L}({\mathbb R}^3),{\mathbb G})$ in $(T{\mathbb R}^3,G)$. 
\end{theo}

The paper is organized as follows. In the next section the new geometry is introduced in the context of almost K\"ahler and almost para-K\"ahler structures. Theorem \ref{t:oneofthema} is proven in section \ref{s3}, while proofs of Theorems \ref{t:geodesics}, \ref{t:ricciflat}, \ref{t:oneofthemainthe} and \ref{t:sourcefields} are contained in sections \ref{s4.1}, \ref{s4.2}, \ref{s4.3}, and \ref{s4.4}, respectively. Theorem \ref{t:minimalembed} is proven in the final section.

\vspace{0.2in}

\section{Almost (para-) K\"ahler structures on TN}\label{s:intr}

Let $N$ be an $n$-dimensional differentiable manifold and $\pi:TN\rightarrow N$ be the canonical projection from the tangent bundle $TN$ to $N$. We define the  {\it vertical bundle} $VN$ as the subbundle ${\mbox Ker}(d\pi)$ of $TTN$. 
If $N$ is equipped with an affine connection $D$, then we may define the {\it horizontal bundle} $HN$ of $TTN$ as follows: 

If $\bar X$ is a tangent vector of $TN$ at $(p_0,V_0)$, there exists a curve $a(t)=(p(t),V(t))\subset TN$ such that $a(0)=(p_0,V_0)$ and $a'(0)=\bar X$. Define the connection map (see \cite{Dombrowski} and \cite{Kowalski}, for further details) $K:TTN\rightarrow TN$ by $K\bar X=D_{p'(0)}V(0)$ if $\bar X\neq VN$ (i.e., $p'(0)\neq 0$) and  if $\bar X\in VN$ (in this case, $\bar X$ is said to be a {\it vertical vector field}) then $K\bar X=V'(0)$. The horizontal bundle $HN$ is simply ${\mbox Ker}(K)$ and therefore we obtain the direct sum:
\begin{eqnarray}
TTN=HN\oplus VN&\simeq& TN\oplus TN\nonumber \\
\bar X&\simeq& (\Pi\bar X, K\bar X),\nonumber
\end{eqnarray}
\begin{prop}\label{p:decomposition}\cite{Kowalski}
Given a vector field $X$ on $(N,D)$ there exist unique vector fields $X^h,X^v$ on $TN$ such that $(\Pi X^h,K X^h)=(X,0)$ and $(\Pi X^v,K X^v)=(0,X)$.  

In addition, if $X,Y$ are vector fields of $N$, we have at $(p,V)\in TN$:
\[
[X^v,Y^v]=0,\quad
[X^h,Y^v]=(D_XY)^v\simeq (0,D_XY),\quad
[X^h,Y^h]\simeq ([X,Y],-R(X,Y)V),
\]
where $R$ denotes the curvature of $D$.
\end{prop}

If $g$ is a Riemannian metric on $N$, we identify the cotangent bundle $T^{\ast} N$ with $TN$ by the following bundle isomorphism: 
\[
\mathfrak g (p,X)=g_{p}(X,.)\quad\mbox{for\;any}\;\; X\in T_pN.
\]
Using the canonical projection $\pi^{\ast}:T^{\ast}N\rightarrow N$, we define the {\it Liouville form} $\xi\in\Omega^1(T^{\ast}N)$ by:
\[
\xi_{(p,\beta)}(\eta)=\beta(d\pi^{\ast}\eta)\quad\mbox{where,}\;\; \beta\in T^{\ast}_pN\;\;\mbox{and}\;\; \eta\in T_{(p,\beta)}T^{\ast}N.
\]
The derivative of the Liouville form defines a canonical symplectic structure, $\Omega_{\ast}:=-d\xi$, on $T^{\ast}N$ and using the isomorphism $\mathfrak g$ we define a symplectic structure $\Omega$ on $TN$ by $\Omega=\mathfrak g^{\ast}\Omega_{\ast}$.
The symplectic structure $\Omega$ is given by
\[
\Omega(\bar X,\bar Y)=g(K\bar X,\Pi\bar Y)-g(\Pi\bar X,K\bar Y).
\]
An {\it almost complex structure} (resp. {\it almost paracomplex structure} ) in $TN$ is an endomorphism $J$ of $TTN$ such that $J^2\bar X=-\bar X$ (resp. $J^2\bar X=\bar X$ and $J$ is not the identity), for every $\bar X\in TTN$. and is said to be {\it compatible} with $\Omega$ if, $\Omega(J.,J.)=\Omega(.,.)$ (resp. $\Omega(J.,J.)=-\Omega(.,.)$ ).
\begin{prop}\label{p:thespaceofcomple}

Let $(N,g)$ be a Riemannian manifold and let $J_0,J_1,J_2$ be the following $(1,1)$-tensors in $TN$:
\[
J_0\bar X\simeq (K\bar X,\Pi\bar X),\qquad J_1\bar X\simeq (\Pi\bar X,-K\bar X),\qquad J_2\bar X\simeq (-K\bar X,\Pi \bar X).
\]
Then $J_0,J_1$ are almost paracomplex structures on $TN$ while, $J_2$ is an almost complex structure all compatible with $\Omega$. The trio $(J_0,J_1,J_2)$ defines an almot para-quartenionic structure on $TN$.
\end{prop}
\begin{proof} A straightforward computation shows that 
\[
J_0^2=J_1^2=\mbox{Id},\qquad J_2^2=-\mbox{Id},\qquad J_0J_1=J_2,
\]
and for any $k\neq l\in\{0,1,2\}$, we have
\[
J_kJ_l=-J_lJ_k.
\]
Furthemore, we have
\[
\Omega(J_0.,J_0.)=\Omega(J_1.,J_1.)=-\Omega(J_2.,J_2.)=-\Omega(.,.),
\]
which shows the compatibility conditions.
\end{proof}

Consider now the metrics $G_0, G_1$ and $G_2$, defined by
\[
G_k(.,.):=\Omega(.,J_k.),\qquad k=0,1,2.
\]
The metric $G_2$ is the Riemannian Sasaki metric
while, $G_0$ is the neutral Sasaki metric. The neutral metric $G_1$ is given by
\begin{equation}{\label{e:theneutralmetric}}
G_1(\bar X,\bar Y)=g(\Pi\bar X,K\bar Y)+g(K\bar X,\Pi\bar Y).
\end{equation}

The Sasaki metrics $G_0$ and $G_2$ are very well known and have been studied extensively by several authors - see for example \cite{BV,gudkappos,yano}.
 In this article we fill the gap by studying the geometry of $(TN,G_1)$. From now on and throughout this article, we simply write $G$ for $G_1$.

\vspace{0.2in}

\section{Curvature of the neutral metric G}\label{s3}

Consider the neutral metric $G$ constructed in Section \ref{s:intr}. We now study the main geometric properties of $(TN,G)$.

Denote the Levi-Civita connection of $G$ by $\nabla$. For a vector field $X$ in $N$ we use Proposition \ref{p:decomposition}, to consider the unique vector fields $X^h$ and $X^v$ in $TN$ such that $\Pi X^h=X$, $KX^h=0$ and $\Pi X^v=0$, $KX^v=X$. We do the same for the vector fields $Y,Z$ on $N$. Since all quantities of type $G(Y^h,Z^v)$ are constant on the fibres, we have that $X^vG(Y^h,Z^v)=0$.

Using the Koszul formula:
\begin{eqnarray}
2G(\nabla_{\bar X}\bar Y,\bar Z)&=&\bar X G(\bar Y,\bar Z)+\bar Y G(\bar X,\bar Z)-\bar Z G(\bar X,\bar Y) +G([\bar X,\bar Y],\bar Z )\nonumber \\
&& -G([\bar X,\bar Z],\bar Y)-G([\bar Y,\bar Z],\bar X),\nonumber
\end{eqnarray}
one may obtain the following relations
\begin{equation}\label{e:equation00}
\nabla_{X^v}Y^v=\nabla_{X^v}Y^h=0,\qquad \nabla_{X^h}Y^v\simeq (0,D_XY).
\end{equation}
We now have at $(p,V)\in TN$,
\begin{eqnarray}
2G(\nabla_{X^h}Y^h,Z^h)&=&X^h G(Y^h,Z^h)+Y^h G(X^h,Z^h)-Z^h G(X^h,Y^h) +G([X^h,Y^h],Z^h )\nonumber \\
&& -G([X^h,Z^h],Y^h)-G([Y^h,Z^h],X^h)\nonumber\\
&=& -g(R(X,Y)V,Z)+g(R(X,Z)V,Y)+g(R(Y,Z)V,X)\nonumber
\end{eqnarray}
and using the first Bianchi identity we finally get, 
\begin{equation}\label{e:equation1}
G(\nabla_{X^h}Y^h,Z^h)=g(R(V,X)Y,Z).
\end{equation}
Similar calculations give,
\begin{equation}\label{e:equation2}
G(\nabla_{X^h}Y^h,Z^v)=g(D_XY,Z).
\end{equation}
Using (\ref{e:equation1}) and (\ref{e:equation2}) we obtain 
\begin{equation}\label{e:equation000}
\nabla_{X^h}Y^h(p,V)\simeq (D_XY,R(V,X)Y),
\end{equation}
where $X=\Pi X^h$ and $Y=\Pi Y^h$.
Putting all together we find that
\begin{equation}\label{e:Levi-Civita}
\Pi\nabla_{\bar X}\bar Y(x,V)=D_{\Pi\bar X}\Pi\bar Y,\qquad K\nabla_{\bar X}\bar Y(x,V)=D_{\Pi\bar X}K\bar Y+R(V,\Pi\bar X)\Pi \bar Y.
\end{equation}

\vspace{0.1in}

\begin{prop}\label{p:riemanncur}
The Riemann curvature tensor $\overline{Rm}$ of the metric $G$ is given by 
\begin{eqnarray}
\overline{Rm}(\bar X,\bar Y, \bar Z, \bar W)|_{(p,V)}&=&\mbox{Rm}(K\bar X,\Pi\bar Y,\Pi\bar Z, \Pi\bar W)+\mbox{Rm}(\Pi\bar X, K\bar Y, \Pi\bar Z, \Pi\bar W)\nonumber\\
&&+\mbox{Rm}(\Pi\bar X,\Pi\bar Y,K\bar Z,\Pi\bar W)+\mbox{Rm}(\Pi\bar X, \Pi\bar Y,\Pi\bar Z,K\bar W)\nonumber \\
&&\qquad\qquad+g((D_VR)(\Pi\bar X,\Pi\bar Y)\Pi\bar Z,\Pi\bar W),\nonumber
\end{eqnarray}
where $Rm$ is the Riemann curvature tensor of $g$ and
\[
(D_uR)(v,w)(z)=D_uR(v,w)z-R(D_uv,w)z-R(v,D_uw)z-R(v,w)D_uz
\]
\end{prop}
 \begin{proof}
 Using the second Bianchi identity, we have at the point $(p,V)\in TN:$
 \begin{eqnarray}
 \overline R(\bar X,\bar Y)\bar Z&=&\nabla_{\bar X}\nabla_{\bar Y}\bar Z-\nabla_{\bar Y}\nabla_{\bar X}\bar Z-\nabla_{[\bar X,\bar Y]}\bar Z\nonumber\\
&\simeq&(D_{\Pi\bar X}D_{\Pi\bar Y}\Pi\bar Z-D_{\Pi \bar Y}D_{\Pi \bar X}\Pi\bar Z-D_{[\Pi\bar X,\Pi\bar Y]}\Pi\bar Z,\nonumber\\
&&\quad D_{\Pi\bar X}D_{\Pi\bar Y}K\bar Z-D_{\Pi \bar Y}D_{\Pi \bar X}K\bar Z-D_{[\Pi\bar X,\Pi\bar Y]}K\bar Z+D_{\Pi\bar X}R(V,\Pi\bar Y)\Pi\bar Z\nonumber\\
&& \qquad+R(V,\Pi\bar X)D_{\Pi\bar Y}\Pi\bar Z)-D_{\Pi\bar Y}R(V,\Pi\bar X)\Pi\bar Z-R(V,\Pi\bar Y)D_{\Pi\bar X}\Pi\bar Z\nonumber \\
&& \qquad\qquad\qquad-R(V,D_{\Pi\bar X}\Pi\bar Y)\Pi\bar Z+R(V,D_{\Pi\bar Y}\Pi\bar X)\Pi\bar Z)\nonumber\\
&=& (R(\Pi\bar X,\Pi\bar Y)\Pi\bar Z,\,R(\Pi\bar X,\Pi\bar Y)K\bar Z +(D_VR)(\Pi\bar X,\Pi\bar Y)(\Pi\bar Z)\nonumber\\
&&\qquad\qquad+R(D_{\Pi\bar X}V,\Pi\bar Y)\Pi\bar Z-R(D_{\Pi\bar Y}V,\Pi\bar X)\Pi\bar Z).\nonumber
\end{eqnarray}
The proposition follows by using the fact that
\[
Rm(X,Y,Z,W)=g(R(X,Y)Z,W).
\]
\end{proof}

\vspace{0.1in}

We are now in position to calculate the Ricci tensor:
\begin{prop}\label{p:ricci}
The Ricci tensor $\overline{Ric}$ of the metric $G$ is given by 
\[
\overline{Ric}(\bar X,\bar Y)=2Ric(\Pi\bar X,\Pi\bar Y),
\]
where $Ric$ denotes the Ricci tensor of $g$.
\end{prop}
\begin{proof}
Following a similar method as in the proof of the main theorem of \cite{AnaR}, we consider the orthonormal frame $(e_1,\ldots,e_n)$ of $(N,g)$. Define the following frame $(\bar e_1,\ldots\bar e_n,\bar e_{n+1},\ldots \bar e_{2n})$ of $TN$ to be the unique vector fields such that
\[
\Pi\bar e_k=e_k,\qquad K\bar e_k=0
\]
\[
\Pi\bar e_{n+k}=0,\qquad K\bar e_{n+k}=e_k
\]
For $i,j=1,\ldots,n$ we have $G_{ij}=0$, and $G_{i,n+j}=\delta_{ij}$.
Observe that $\overline{Ric}(X^v,Y^v)=\overline{Ric}(X^h,Y^v)=0$. Using Proposition \ref{p:riemanncur}, we have
\begin{eqnarray}
\overline{Ric}(X^h,Y^h)&=&\sum_{i,j=1}^n G^{i,n+j}(\overline{Rm}(X^h,\bar e_i,Y^h,\bar e_{n+j})+\overline{Rm}(X^h,\bar e_{n+j},Y^h,\bar e_i))\nonumber\\
&=&\sum_{i,j=1}^n \delta_{ij}(\overline{Rm}(X^h,\bar e_i,Y^h,\bar e_{n+j})+\overline{Rm}(X^h,\bar e_{n+j},Y^h,\bar e_i))\nonumber\\
&=&\sum_{i,j=1}^n \delta_{ij}(Rm(X,e_i,Y,e_{j})+Rm(X,e_{j},Y,e_i))\nonumber\\
&=&2\sum_{i=1}^n Rm(X,e_i,Y,e_i)\nonumber\\
&=&2Ric(X,Y),\nonumber
\end{eqnarray}
and this completes the Proposition. 
\end{proof}

\vspace{0.1in}

\begin{proof}[{\bf Proof of Theorem \ref{t:oneofthema}}]

1. For the first part consider, as before, the frame $(\bar e_1,\ldots\bar e_n,\bar e_{n+1},\ldots \bar e_{2n})$ of $TN$. The scalar curvature $\bar S$ of the metric $G$ is
\[
\bar S=\sum_{i,j=1}^{2n}G^{ij}\overline{Ric}(\bar e_i,\bar e_j)=\sum_{i,j=1}^{2n}G^{ij}Ric(\Pi\bar e_i,\Pi\bar e_j)=0.
\]
2. The second part follows directly from the Proposition \ref{p:ricci}. 

3. We now prove the third part. Let $\overline W$ be the Weyl tensor of $G$. Using Proposition 5 of \cite{AnaR} we have
\[
\overline W(\bar X,\bar Y,\bar Z,\bar W)=\overline{Rm}(\bar X,\bar Y,\bar Z,\bar W)-\frac{\mbox{Ric}(\Pi\bar Y,\Pi\bar W)G(\bar X,\bar Z)+\mbox{Ric}(\Pi\bar X,\Pi\bar Z)G(\bar Y,\bar W)}{n-1}
\]
\begin{equation}\label{e:weyltensor}
\qquad\qquad +\frac{\mbox{Ric}(\Pi\bar Y,\Pi\bar W)G(\bar X,\bar Z)+\mbox{Ric}(\Pi\bar X,\Pi\bar Z)G(\bar Y,\bar W)}{n-1}.
\end{equation}
Assume that $G$ is locally conformally flat and $n\geq 3$. We then have
\[
\overline{Rm}(\bar X,\bar Y,\bar Z,\bar W)=\frac{1}{n-1}\left(\mbox{Ric}(\Pi\bar Y,\Pi\bar W)G(\bar X,\bar Z)+\mbox{Ric}(\Pi\bar X,\Pi\bar Z)G(\bar Y,\bar W)\right)
\]
\begin{equation}\label{e:riemcurrel}
-\frac{1}{n-1}\left(\mbox{Ric}(\Pi\bar Y,\Pi\bar Z)G(\bar X,\bar W)+\mbox{Ric}(\Pi\bar X,\Pi\bar W)G(\bar Y,\bar Z)\right)
\end{equation}
Let $X,Y,Z,W$ be vector fields on $N$ with corresponding unique vector fields $X^h,X^v$, $Y^h,Y^v,Z^h,Z^v,W^h,W^v$ on $TN$.
Using Proposition \ref{p:riemanncur}, we have
\[
\overline{Rm}(X^h,Y^h,Z^h,W^v)=Rm(X,Y,Z,W),
\]
and thus, (\ref{e:riemcurrel}) becomes,
\[
Rm(X,Y,Z,W)=-\frac{1}{n-1}(\mbox{Ric}(X,Z)g(Y,W)-\mbox{Ric}(Y,Z)g(X,W))
\]
We now have,
\[
Rm(X,Y,X,Y)=-\frac{1}{n-1}(\mbox{Ric}(X,X)|Y|^2-\mbox{Ric}(X,Y)g(X,Y)).
\]
On the other hand,
\[
Rm(Y,X,Y,X)=-\frac{1}{n-1}(\mbox{Ric}(Y,Y)|X|^2-\mbox{Ric}(X,Y)g(X,Y)).
\]
which implies
\[
\frac{\mbox{Ric}(X,X)}{|X|^2}=\frac{\mbox{Ric}(Y,Y)}{|Y|^2},
\]
for any vector fields $X,Y$ on $N$. Thus, there exists a smooth function $\lambda$ on $N$ such that
\[
\mbox{Ric}(X,X)=\lambda |X|^2, 
\]
which shows that $g$ is Einstein. Since $n\geq 3$, the function $\lambda$ must be constant. Let $P$ be the plane spanned by $\{e_1,e_2\}$. Then the sectional curvature 
\[
K(P)=Rm(e_1,e_2,e_2,e_1)=\frac{\lambda}{n-1},
\]
which is constant. 

Assume the converse, that is, $n\geq 3$ and $g$ is of constant sectional curvature $K$. Thus, 
\begin{equation}\label{e:sectionalscalar}
K=\frac{R}{n(n-1)},
\end{equation}
where $R$ denotes the scalar curvature. Also $g$ is locally symmetric and therefore from Proposition \ref{p:riemanncur} we have
\begin{eqnarray}
\overline{Rm}(\bar X,\bar Y, \bar Z, \bar W)&=&\mbox{Rm}(K\bar X,\Pi\bar Y,\Pi\bar Z, \Pi\bar W)+\mbox{Rm}(\Pi\bar X, K\bar Y, \Pi\bar Z, \Pi\bar W)\nonumber\\
&&+\mbox{Rm}(\Pi\bar X,\Pi\bar Y,K\bar Z,\Pi\bar W)+\mbox{Rm}(\Pi\bar X, \Pi\bar Y,\Pi\bar Z,K\bar W)\nonumber 
\end{eqnarray}
Hence, $\overline{Rm}(X^h,Y^h,Z^h,W^h)=0$ and using (\ref{e:weyltensor}) we have,
\[
\overline W(X^h,Y^h,Z^h,W^h)=-\frac{1}{n-1}\left(\mbox{Ric}(Y,W)G(X^h,Z^h)+\mbox{Ric}(X,Z)G(Y^h,W^h)\right)
\]
\[
+\frac{1}{n-1}\left(\mbox{Ric}(Y,Z)G(X^h,W^h)+\mbox{Ric}(X,W)G(Y^h,Z^h)\right)=0.
\]
Also,
\[
\overline W(X^h,Y^h,Z^h,W^v)=\overline{Rm}(X^h,Y^h,Z^h,W^v)-\frac{1}{n-1}\left(\mbox{Ric}(X,Z)G(Y^h,W^v)\right)
\]
\[
+\frac{1}{n-1}\left(\mbox{Ric}(Y,Z)G(X^h,W^v)\right)
\]
\[
=Rm(X,Y,Z,W)-\frac{1}{n-1}\left(\mbox{Ric}(X,Z)g(Y,W)
-\mbox{Ric}(Y,Z)G(X,W)\right)
\]
\[
=Rm(X,Y,Z,W)-\frac{R}{n(n-1)}\left(g(X,Z)g(Y,W)
-g(Y,Z)G(X,W)\right).
\]
Since $g$ is of constant sectional curvature $K$, we have
\[
Rm(X,Y,Z,W)=K(g(X,Z)g(Y,W)
-g(Y,Z)G(X,W)),
\]
and therefore, using (\ref{e:sectionalscalar}) we have
\[\overline W(X^h,Y^h,Z^h,W^v)=\left(K-\frac{R}{n(n-1)}\right)\left(g(X,Z)g(Y,W)
-g(Y,Z)G(X,W)\right)=0.
\]
Similarly, we prove that all coefficients of the Weyl tensor vanish and  that $\overline W=0$. Thus, for $n\geq 3$, the metric $G$ is locally conformally flat if and only if $g$ is of constant sectional curvature. 

For $n=2$, the Riemann curvature tensor is given by
\[
Rm(X,Y,Z,W)=K(g(X,Z)g(Y,W)
-g(Y,Z)G(X,W)),
\]
where $K$ is the Gauss curvature of $g$. Hence, following a similar argument as before, one can prove that for every 2-manifold $(N,g)$ the neutral metric $G$ of $TN$ is locally conformally flat.

4. We now proceed with the last part of the proof. Assume first that $g$ is locally symmetric. Then for any vector fields $\xi, X,Y,Z$ on $N$ we have, by definition, 
\begin{equation}\label{e:localsymmetry}
D_{\xi}(R(X,Y)Z)=R(D_{\xi}X,Y)Z+R(X,D_{\xi}Y)Z+R(X,Y)D_{\xi}Z.
\end{equation}
Using (\ref{e:localsymmetry}), a brief computation shows
\[
R(\xi,V)(R(X,Y)Z)=R(R(\xi,V)X,Y)Z+R(X,R(\xi,V)Y)Z+
\]
\begin{equation}\label{e:localsymmetry1}
\qquad\qquad +R(X,Y)(R(\xi,V)Z).
\end{equation}
We want to prove that $G$ is locally symmetric, that is, $\nabla\bar{R}=0$. Proposition \ref{p:riemanncur} tells us that
\[
\bar R(\bar X,\bar Y) \bar Z|_{(p,V)} \simeq (R(\Pi\bar X,\Pi\bar Y)\Pi\bar Z,R(K\bar X,\Pi\bar Y)\Pi\bar Z+R(\Pi\bar X,K\bar Y)\Pi\bar Z+R(\Pi\bar X,\Pi\bar Y)K\bar Z
\]
\[
+(D_VR)(\Pi\bar X,\Pi\bar Y)\Pi\bar Z),
\]
and using the fact that $g$ is locally symmetric we have
\[
\bar R(\bar X,\bar Y) \bar Z\simeq (R(\Pi\bar X,\Pi\bar Y)\Pi\bar Z,R(K\bar X,\Pi\bar Y)\Pi\bar Z+R(\Pi\bar X,K\bar Y)\Pi\bar Z+R(\Pi\bar X,\Pi\bar Y)K\bar Z.
\]
We thus obtain the following
\[
\bar R(X^v,Y^v)Z^v=\bar R(X^v,Y^v)Z^h=\bar R(X^h,Y^v)Z^v=0,
\]
\[
\bar R(X^v,Y^h)Z^h\simeq (0,R(X,Y)Z),
\]
and 
\[
\bar R(X^h,Y^h)Z^h\simeq (R(X,Y)Z,0).
\]
Applying the relations (\ref{e:equation00}) and (\ref{e:equation000}) we get
\[
\nabla_{\xi^v}(\bar R(X^v,Y^h)Z^h)=\nabla_{\xi^h}(\bar R(X^v,Y^h)Z^h)=0,
\]
\[
\nabla_{\xi^v}(\bar R(X^v,Y^h)Z^h)\simeq (0,D_{\xi}(R(X,Y)Z)),
\]
\[
\nabla_{\xi^h}(\bar R(X^h,Y^h)Z^h)\simeq (D_{\xi}(R(X,Y)Z),R(V,\xi)(R(X,Y)Z)),
\]
\[
\bar R(\nabla_{\xi^h}X^h,Y^h)Z^h\simeq (R(D_{\xi}X,Y)Z,R(R(V,\xi)X,Y)Z),
\]
\[
\bar R(X^h,\nabla_{\xi^h}Y^h)Z^h\simeq (R(X,D_{\xi}Y)Z,R(X,R(V,\xi)Y)Z),
\]
\[
\bar R(X^h,Y^h)\nabla_{\xi^h}Z^h\simeq (R(X,Y)D_{\xi}Z,R(X,Y)R(V,\xi)Z),
\]
We now use all relations above and together with (\ref{e:localsymmetry1}) finally obtain,
\begin{eqnarray}
(\nabla_{\xi^h}\bar R)(X^h,Y^h,Z^h)&=&\nabla_{\xi^h}(\bar R(X^h,Y^h)Z^h)-\bar R(\nabla_{\xi^h}X^h,Y^h)Z^h-\bar R(X^h,\nabla_{\xi^h}Y^h)Z^h\nonumber \\
&&\qquad\qquad\qquad\qquad\qquad -\bar R(X^h,Y^h)\nabla_{\xi^h}Z^h\nonumber \\
&=& 0.\nonumber
\end{eqnarray}
Similar arguments shows that this relation holds:
\[
(\nabla_{\xi^v}\bar R)(X^h,Y^h,Z^h)=0,
\]
showing that $\nabla\bar R=0$, which means  $G$ is locally symmetric.

Conversely, assume that $G$ is locally symmetric. Then the following holds true:
\[
\nabla_{\xi^h}(\bar R(X^h,Y^h)Z^v)=\bar R(\nabla_{\xi^h}X^h,Y^h)Z^v+\bar R(X^h,\nabla_{\xi^h}Y^h)Z^v+\bar R(X^h,Y^h)\nabla_{\xi^h}Z^v,
\]
implying, 
\[
D_{\xi}(R(X,Y)Z)=R(D_{\xi}X,Y)Z+R(X,D_{\xi}Y)Z+R(X,Y)D_{\xi}Z,
\]
which means that $g$ is locally symmetric, completing the proof of the Theorem.
\end{proof}

\vspace{0.1in}

Suppose that the manifold $N$ is equipped with a K\"ahler structure $(j,g)$. An almost complex structure ${\mathbb J}$ on $TN$ can be defined by 
\[
{\mathbb J}\bar X=(j\Pi\bar X,jK\bar X).
\]
It has been proved that ${\mathbb J}$ is integrable and one can check easily that is compatible with $G$, that is,
\[
G({\mathbb J}.,{\mathbb J}.)=G(.,.).
\]
It can be easily proved that ${\mathbb J}$ is also parallel with respect to $\nabla$. Namely,
\begin{eqnarray}
\nabla_{\bar X}{\mathbb J}\bar Y(p,V)&=&(D_{\Pi\bar X}j\Pi\bar Y,D_{\Pi\bar X}jK\bar Y+R(V,\Pi\bar X)j\Pi \bar Y)\nonumber \\
&=&(jD_{\Pi\bar X}\Pi\bar Y,j(D_{\Pi\bar X}K\bar Y+R(V,\Pi\bar X)\Pi \bar Y))\nonumber \\
&=&{\mathbb J}(D_{\Pi\bar X}\Pi\bar Y,D_{\Pi\bar X}K\bar Y+R(V,\Pi\bar X)\Pi \bar Y)\nonumber \\
&=&{\mathbb J}\nabla_{\bar X}\bar Y(p,V).\nonumber
\end{eqnarray}


The complex structure ${\mathbb J}$ is compatible with $\Omega$ and together with the metric ${\mathbb G}$ given by
\[
{\mathbb G}=\Omega({\mathbb J}.,.)
\]
defines another K\"ahler structure $({\mathbb G},\Omega,{\mathbb J})$ on $TN$ which it has been introduced by H. Anciaux and R. Pascal in \cite{AnaR}. In particular,
\begin{equation}\label{e:henrimetric}
{\mathbb G}(\bar X,\bar Y)=g(K\bar X,J\Pi\bar Y)-g(\Pi\bar X,JK\bar Y).
\end{equation}




The following Proposition shows that the neutral metric $G$ is an extension of ${\mathbb G}$ for the non-K\"ahler structures.


\begin{prop}\label{p:isometricmetrics}
The metrics $G$ and ${\mathbb G}$, defined respectively in (\ref{e:theneutralmetric}) and
(\ref{e:henrimetric}), are isometric.
\end{prop}
\begin{proof}
Let $N$ be a smooth manifold equipped with a K\"ahler structure $(j,g)$. Let $G$ and ${\mathbb G}$ be the K\"ahler metrics defined as above and define the following diffeomorphism:
\[
f:TN\rightarrow TN:(p,V)\mapsto (p,-jV).
\]
If $\bar X\in T_{(p,V)}TN$ then $\Pi df(\bar X)=\Pi\bar X$ and using the fact that $j$ is parallel, we have that $Kdf(\bar X)=-jK\bar X$. Thus,
\begin{eqnarray}
f^{\ast}G(\bar X,\bar Y)&=&g(\Pi\bar X,-jK\bar Y)+g(\Pi\bar Y,-jK\bar X)\nonumber \\
&=&-g(\Pi\bar X,jK\bar Y)+g(j\Pi\bar Y,K\bar X)\nonumber \\
&=& {\mathbb G}(\bar X,\bar Y),\nonumber
\end{eqnarray}
which shows that $f$ is an isometry.
\end{proof}

\section{Submanifold theory}

We now investigate the submanifold theory of $(TN,G)$. In particular, we will study geodesics and the Lagrangian graphs. 

\subsection{Geodesics}\label{s4.1}
We are now in position to characterize the geodesics of the neutral metric $G$. In fact, we prove our second result:

\begin{proof}[\bf Proof of Theorem \ref{t:geodesics}]
Let $\bar X(t):=\gamma'(t)$. Then, using (\ref{e:Levi-Civita}), we have $\Pi \nabla_{\gamma'}\gamma'=D_{\Pi\bar X}\Pi\bar X$, and $K\nabla_{\gamma'}\gamma'=D_{\Pi\bar X}K\bar X+R(V,\Pi\bar X)\Pi\bar X$. 
If $\gamma(t)=(x(t),V(t))$  is a geodesic then, $D_{\Pi\bar X}\Pi\bar X=0$, and thus $D_{x'}x'=0$, which implies that $x(t)$ is a geodesic. On the other hand, $K\bar X=D_{x'}V$ and therefore, $D_{\Pi\bar X}K\bar X=D^2_{x'}V$. 
Using the fact that $D_{\Pi\bar X}K\bar X+R(V,\Pi\bar X)\Pi\bar X=0$, we have $D^2_{x'}V+R(V,x')x'=0$, which shows that $V$ is a Jacobi field along the geodesic $x(t)$.

Conversely, when $V(t)$ is a Jacobi field along the geodesic $x(t)$ then, obviously, $\gamma(t)=(x(t),V(t))$  is a geodesic.
\end{proof}

\vspace{0.1in}

\subsection{Graph submanifolds}\label{s4.2}

Let $(N,g)$ be a $n$-dimensional Riemannian manifold and $U$ be an open subset of $N$. Considering a vector field $V$ in $N$, we obtain a $n$-dimensional submanifold ${\mathbb V}\subset TN$ which is a section of the canonical bundle $\pi: TN\rightarrow N$. Such submanifolds are immersed as graphs, that is, ${\mathbb V}=f(N)$, where $f(p)=(p,V(p))$. 

The following proposition gives a relation between the null points of the graph ${\mathbb V}$ with the critical points 
of the length function of an integral curve of $V$.

\begin{prop}
Let $V$ be a vector field of $N$ and ${\mathbb V}$ be the corresponding graph. Then the following two statements hold true:
\begin{enumerate}
\item If an integral curve of $V$ is a geodesic in $(N,g)$ then  ${\mathbb V}$ admits a null curve.
\item If $V$ admits a closed integral curve then ${\mathbb V}$ must contain a null point.
\end{enumerate}
\end{prop}
\begin{proof}
Let $p=p(t)$ be an integral curve of $V$, that is, 
$V(t):=V(p(t))=p'(t)$.
The corresponding curve in $TN$ is given by
\[
f(t)=(p(t),V(t)).
\]
Then $f'(t)=(p'(t),D_{p'(t)}V(t))$ and thus,
\begin{eqnarray}
G(f'(t),f'(t))&=&2g(p'(t),D_{p'(t)}V(t))=2g(p'(t),D_{p'(t)}p'(t))\nonumber \\
&=&D_{p'(t)}\Big(g(p'(t),p'(t))\Big)=\frac{d}{dt}|p'|^2=\frac{d}{dt}|V|^2.\nonumber 
\end{eqnarray}

(1) When the curve $p=p(t)$ is a geodesic then $D_{p'(t)}p'(t)$ must vanish and thus $G(f'(t),f'(t))=0$ for any $t$. Therefore $f(t)$ is a null curve.

(2) Assuming that the integral curve $p(t)$ is closed, there exists $t_0\in{\mathbb S}^1$ such that,
\[
\frac{d}{dt}|p'(t_0)|^2=\frac{d}{dt}|V(t_0)|^2=0,
\]
which means that $f$ is null at the point $t_0$.
\end{proof}
\vspace{0.1in}

We now study Lagrangian graphs in $TN$. We need first to recall the definition of a Lagrangian submanifold:

\begin{defi}
Let $N$ be a $2n$-dimensional manifold equipped with a symplectic structure $\omega$. An immersion $f:\Sigma^n\rightarrow N$ is said to be {\it Lagrangian} if $f^{\ast}\omega = 0$.
\end{defi}
\vspace{0.1in}

For Lagrangian graphs we have the following:
\vspace{0.1in}

\begin{prop}\label{p:lagrangraphgradi}
Let $V$ be a vector field in the open subset $U\subset N$. The graph ${\mathbb V}$ is Lagrangian if and only if $V$ is locally the gradient of a real smooth function $u$ on $U\subset N$, i.e, $V=D u$.
\end{prop}

We now prove our third result.

\begin{proof}[{\bf Proof of Theorem \ref{t:ricciflat}}]
Let $g$ be the Riemannian metric in $N$ and ${\mathbb V}$ be the submanifold of $TN$ obtained by the image of the graph:
\[
f:U\subset N\rightarrow TN:p\mapsto (p,V(p)),
\]
where $V$ is a vector field defined on the open subset of $N$. 
The fact that $f$ is Lagrangian implies that the almost paracomplex structure $J$ is a bundle isomorphism between the tangent bundle $T{\mathbb V}$ and the normal bundle $N{\mathbb V}$. We then consider the Maslov form $\eta$ on ${\mathbb V}$ defined by, 
\[
\eta=G(J{\mathbb H},.),
\]
where ${\mathbb H}$ is the mean curvature vector of $f$. The Lagrangian condition implies the following relation:
\[
d\eta=\frac{1}{2}\overline{Ric}(J.,.)|_{{\mathbb V} },
\]
where $\overline{Ric}$ denotes the Ricci tensor of $G$. Assuming that ${\mathbb H}$ is parallel, the Maslov form is closed and therefore,
\[
\overline{Ric}(J\bar X,\bar Y)=0,
\]
for every tangential vector fields $\bar X,\bar Y$. If $Ric$ denotes the Ricci tensor of $g$, the Proposition \ref{p:ricci}, gives
\[
\overline{Ric}(\bar X,\bar Y)=2Ric(\Pi\bar X,\Pi\bar Y),
\]
If $X,Y$ are vector fields in $U$, the fact that $f$ is a graph, implies that $\Pi df(X)=X$ and $\Pi df(Y)=Y$. On the other hand, using the definition $J$, we have
\[
J(df(X))=(\Pi df(X),-K df(X)).
\]
Thus,
\begin{eqnarray}
0&=&\overline{Ric}(Jdf(X),df(Y))\nonumber\\
&=&\overline{Ric}((\Pi df(X),-Kdf(X)),(\Pi df(Y),Kdf(Y))\nonumber\\
&=&2Ric(\Pi df(X),\Pi df(Y)\nonumber\\
&=&2Ric(X,Y),\nonumber
\end{eqnarray}
and the Theorem follows.
\end{proof}

\vspace{0.1in}

\begin{coro}
Let $(N,g)$ be a non-flat Riemannian 2-manifold and $\Sigma$ be a Lagrangian surface of $(TN, G,\Omega)$. Then $\Sigma$ has parallel mean curvature if and only if it is a set of lines that are orthogonal to a  geodesic $\gamma$ of $N$.
\end{coro}
\begin{proof}
Suppose that $g$ is non-flat. Using Theorem \ref{t:ricciflat}, the Lagrangian surface $\Sigma$ can't be the graph of a smooth function on $N$. Following a similar argument as the proof of Proposition 2.1 of \cite{An}, $\Sigma$ can be parametrized by:
\[
f:U\subset {\mathbb R}^2\rightarrow TN:(s,t)\mapsto (\gamma(s),a(s)\gamma'(s)+tj\gamma'(s)),
\]
where $j$ denotes the canonical complex structure on $N$ defined as a rotation on $TN$ about $\pi/2$ and $\gamma=\gamma(s)$ is a curve in $N$. The mean curvature ${\mathbb H}$ of $f$ is
\[
{\mathbb H}=(0,k(s)j\gamma'(s)),
\]
where $k$ denotes the curvature of $\gamma$. Obviously, we have that $\nabla_{\partial_t} {\mathbb H}=0$ and 
\[
\nabla_{\partial_s} {\mathbb H}=(0,-k^2\gamma'+k_sj\gamma'),
\]
which shows that $\Sigma$ has parallel mean curvature is equivalent to the fact that $\gamma$ is a geodesic. 
\end{proof}
\vspace{0.1in}

\subsection{Lagrangian graphs in the Euclidean space.}\label{s4.3}

In this subsection we study Lagrangian graphs in $T{\mathbb R}^n$.

\begin{defi}
A smooth function $u$ on ${\mathbb R}^n$ is said to be {\bf functionally related of second order} if for every two pairs $(i_1,i_2)$ and $(j_1,j_2)$ there exists a function $F$ on ${\mathbb R}^2$ such that $F(u_{x_{i_1}x_{i_2}},u_{x_{j_1}x_{j_2}})=0$. 
\end{defi}

\vspace{0.1in}

{\bf Example:} 
The following functions in ${\mathbb R}^n$ are functionally related of second order: $$u(x_1,\ldots,x_n):=f(a_1x_1+\ldots+a_nx_n)$$and $$v(x_1,\ldots,x_n)=\displaystyle\sum_{1\leq i\leq j\leq n}(a_{ij}x_ix_j+b_ix_i+c).$$ 

\vspace{0.1in}

Consider the $n$-dimensional Euclidean space $({\mathbb R}^n,ds^2)$, where $ds^2$ is the usual inner product and let $G$ be the neutral metric in $T{\mathbb R}^n$ derived by $ds^2$.



For Lagrangian graphs in $T{\mathbb R}^{n}$ we prove the Theorem \ref{t:oneofthemainthe}.

\vspace{0.1in}

\begin{proof}[{\bf Proof of Theorem \ref{t:oneofthemainthe}}]
Let $(x_1,\ldots,x_{n})$ be the standard Cartesian coordinates of ${\mathbb R}^{n}$ and let 
\[
f(x_1,\ldots,x_{n})=(x_1,\ldots,x_{n},u_{x_1},\ldots,u_{x_{n}}),
\]
be the local expression of the Lagrangian submanifold. 

1. We then have,
\[
f_i:=f_{x_i}=(\partial/\partial x_i,D_{\partial/\partial x_i}D u)=(\partial/\partial x_i,u_{x_ix_1},\ldots,u_{x_ix_{n}})
\]
The first fundamental form has coefficients:
\[
g_{ij}=G(f_i,f_j)=2\left<\partial/\partial x_i,D_{\partial/\partial x_j}D u\right>=2u_{x_ix_j},
\]
which yields,
\[
g=f^{\ast}G=2\,\mbox{Hess}(u).
\]
For tangential vector fields $X,Y,Z,W$ the Gauss equation for $f$ is 
\begin{equation}\label{e:gaussequ}
G(\bar R(X,Y)Z,W)=G(\bar B(X,W),\bar B(Y,Z))-G(\bar B(X,Z),\bar B(Y,W)),
\end{equation}
where $\bar B$ denotes the second fundamental form of $f$. The Levi-Civita connection $\nabla$ of $g$ is:
\[
\nabla_{f_i}f_j=\sum_{l,k}g^{lk}u_{ijl}f_k.
\]
Thus, 
\begin{equation}\label{e:secondfund}
\bar B(f_i,f_j)=f_{ij}-\sum_{l,k}g^{lk}u_{ijl}f_k.
\end{equation}
The equation (\ref{e:gaussequ}) gives
\[
G(\bar R(f_i,f_j)f_k,f_l)=\sum_{s,t}g^{st}(\partial_{x_k}u_{x_ix_s}\partial_{x_l}u_{x_jx_t}-\partial_{x_l}u_{x_ix_s}\partial_{x_k}u_{x_jx_t})
\]

Assuming now that $u$ is functionally related of second order implies
\[
\partial_{x_k}u_{x_ix_s}\partial_{x_l}u_{x_jx_t}-\partial_{x_l}u_{x_ix_s}\partial_{x_k}u_{x_jx_t}=0,
\]
for any indices $i,j,k,l,s,t$, and this completes the first statement of the Theorem.

\vspace{0.1in}

2. Let $GL(n,{\mathbb R})$ be the space of all invertible $n\times n$ real matrices and let $g:{\mathbb R}^n\rightarrow GL(n,{\mathbb R}):(x_1,\ldots,x_n)\mapsto (g_{ij}(x_1,\ldots,x_n))$ be a smooth mapping. Using Jacobi's formula,
\[
\frac{\partial}{\partial x_l}(\det g)={\mbox Tr}\left(\mbox{Adj}(g)\,\frac{\partial}{\partial x_l}g\right),
\]
where $\mbox{Adj}$ denotes the adjoint matrix of $h$, then one can prove the following:

\begin{lemm}\label{l:derivofmatrix}
Let $g:{\mathbb R}^n\rightarrow GL(n,{\mathbb R})$ be a smooth mapping. If $g_{ij}$ and $g^{ij}$ are the entries of $g$ and $g^{-1}$, respectively, then
\[
\sum_{i,j}g^{ij}\partial_{x_l}g_{ij}=\partial_{x_l}\log|\det g|.
\]
\end{lemm}
Using the expression (\ref{e:secondfund}), the second fundamental form becomes
\begin{eqnarray}
\bar B(f_i,f_j)&=&(0,(u_{1ij},\ldots,u_{nij}))-\sum_{l,k}g^{lk}u_{ijl}(\partial/\partial x_k,(u_{1k},\ldots,u_{nk}))\nonumber \\
&=&\Big(-\sum_{l,k}g^{lk}u_{ijl}\partial/\partial x_k,\frac{1}{2}\sum_{l}u_{ijl}\partial/\partial x_l\Big)\label{e:secfundamentalform}
\end{eqnarray}
The mean curvature vector ${\mathbb H}$ is
\[
{\mathbb H}=\Big(-\sum_{i,j,l,k}g^{ij}g^{lk}u_{ijl}\partial/\partial x_k,\frac{1}{2}\sum_{i,j,l}g^{ij}u_{ijl}\partial/\partial x_l\Big)
\]
A straightforward computation gives
\[
{\mathbb H}=-\sum_{i,j,l,k}g^{ij}g^{lk}u_{ijl}Jf_{k}
\]
We then have from Lemma \ref{l:derivofmatrix},
\begin{equation}\label{e:meancurvintermsj}
J{\mathbb H}= -\sum_{l,k}g^{lk}\partial_{x_l}\log|\det Hess(u)| f_k.
\end{equation}
Using (\ref{e:meancurvintermsj}), we get
\[
G(J{\mathbb H},f_s)=\partial_{x_s}\log|\det Hess(u)|^{-2},
\]
and thus,
\[
J{\mathbb H}=\nabla\log|\det Hess(u)|^{-2}.
\]
Denote the divergence with respect to the induced metric $g$ by $\mbox{div}$. Then,
\[
\mbox{div}J{\mathbb H}=\Delta\log|\det Hess(u)|^{-2}.
\]
where $\Delta$ denotes the Laplacian with respect to $g$. Therefore, $f$ is Hamiltonian minimal if and only if $\Delta\log|\det Hess(u)|=0$.

\vspace{0.1in}

3. The minimal condition follows easily from (\ref{e:meancurvintermsj}). 

Assuming that $f$ is totally geodesic, the relation (\ref{e:secfundamentalform}) gives
\[
u_{ijl}=0,
\]for every indices $i,j,k$ and thus $u$ must satisfy (\ref{e:totallycondition}).

Conversely, assume that $u$ is defined by (\ref{e:totallycondition}). Then all third derivatives vanish and obviously the second fundamental form $B$ is identically zero. The fact that totally geodesic Lagrangian graphs are flat comes from part (1) of the Theorem.
\end{proof}

As an application of Theorem \ref{t:oneofthemainthe}, we give the following corollary:
 
  \begin{coro}
Let $\Omega$ be an embedded ball in ${\mathbb R}^n$ such that $\partial\Omega$ is an embedded sphere. Suppose that $f$ is a Hamiltonian minimal graph in $\Omega$ such that the induced metric is Riemannian. If $f$ is minimal at $\partial\Omega$ then it is minimal in $\Omega$.  
 \end{coro}
\begin{proof}
Let $g$ be the induced metric of $G$ in $\Omega$ through $f$ and let $\eta$ be the outward unit vector field normal to $\partial\Omega$ with respect to $g$. Since $f$ is a Lagrangian graph, there exists a smooth function $u:\Omega\rightarrow {\mathbb R}$ such that $f(p)=(p,Du(p))$, $p\in\Omega$. Therefore, $g=2\mbox{Hess}(u)$ and let $h=\det\mbox{Hess}(u)$. The immersion $f$ is minimal at $\partial\Omega$ and thus the Theorem \ref{t:oneofthemainthe} tells us that $h$ is constant in $\partial\Omega$. Using Stokes Theorem and the fact that $\partial\Omega$ is a sphere we have,
\[
\int_{\Omega}\mbox{div}(h\nabla h)=\int_{\partial\Omega}g(h\nabla h,\eta)=h\int_{\partial\Omega}g(\nabla h,\eta)=h\int_{\partial\Omega}\frac{\partial h}{\partial\eta}=0,
\]
where $\mbox{div}$ and $\nabla$ denote the divergence and the gradient of $g$. This, implies 
\begin{equation}\label{e:eqcoro1}
\int_{\Omega}g(\nabla h,\nabla h)=-\int_{\Omega}\Delta h,
\end{equation}
where $\Delta$ is the Laplacian of $g$.

On the other hand, $f$ is Hamiltonian minimal. Then, Theorem \ref{t:oneofthemainthe} says that $\log h$ must be harmonic. In other words,
\[
\Delta\log h=0.
\]
That means,
\[
g(\nabla h,\nabla h)=h\Delta h,
\]
and by integrating over $\Omega$, we then have
\begin{equation}\label{e:eqcoro2}
\int_{\Omega}g(\nabla h,\nabla h)=\int_{\Omega}\Delta h.
\end{equation}
Using (\ref{e:eqcoro1}) and (\ref{e:eqcoro2}) we have
\[
\int_{\Omega}g(\nabla h,\nabla h)=0,
\]
and since $g$ is Riemannian it follows that
\[
\nabla h(p)=0,
\] 
for every $p$ in $\Omega$ and the Corollary is completed.
\end{proof}

\vspace{0.1in}
\subsection{Source fields in ${\mathbb R}^n$}\label{s4.4}

As an application of Theorem \ref{t:oneofthemainthe}, we explore the geometry of source vector fields in ${\mathbb R}^n$.

\begin{defi}\label{d:source}
A vector field $V$ in ${\mathbb R}^n-\{0\}$ is said to be an {\it $SO(n)$ invariant source field} (or simply a {\it source field}) if it is given by
\[
V=H(R)\frac{\partial}{\partial R}
\]
where $R$ is the distance to the origin (the source). The function $H$ is called the field intensity.
\end{defi}

We now study submanifolds in $T{\mathbb R}^n$ that are graphs of source fields in ${\mathbb R}^n$. If $V$ is the source field, the corresponding graph in ${\mathbb R}^n$ will be denoted by ${\mathbb V}$. 

\begin{prop}
Let $f:{\mathbb R}^n-\{0\}\rightarrow T({\mathbb R}^n-\{0\}):p\mapsto (p,V(p))$ be the graph of a source field $V$ in $\R^n$ with field intensity $H$. Then $f$ is Lagrangian and in particular $V=\nabla u$, where $u=u(R)$ is the anti-derivative of $H$.
\end{prop}
\begin{proof}
Observe that if $u'(R)=H(R)$, then 
\[
\frac{\partial u}{\partial x_i}=\frac{du}{dR}\frac{\partial R}{\partial x_i}=\frac{H(R)}{R}x_i.
\]
Thus,
\begin{eqnarray}
V&=&H(R)\frac{\partial}{\partial R}\nonumber \\
&=&\frac{H(R)}{R}\sum_{i}x_i\frac{\partial}{\partial x_i}\nonumber \\
&=&\sum_{i}u_{x_i}\frac{\partial}{\partial x_i},\nonumber 
\end{eqnarray}
and the proposition follows.
\end{proof}

\vspace{0.1in}

We now prove our next result:

\begin{proof}[{\bf Proof of Theorem \ref{t:sourcefields}}]
Let $H$ be the field intensity of a source field $V$. Then
\[
V=H(R)\frac{\partial}{\partial R}
\]
where $R=\sqrt{x_1^2+\ldots+x_n^2}$. If $h=H/R$, then $V$ can be also written as
\[
V(x_1,\ldots,x_n)=h(R)\,(x_1,\ldots,x_n).
\]
It is not hard to see that there exists a smooth function $u$ such that $V=\nabla u$, which means that $u_{x_i}=h(R)x_i$. The induced metric $g=f^{\ast}G$ has coefficients
\begin{equation}\label{e:thefunctionh}
g_{ij}=2u_{x_ix_j}=2h\left(\frac{x_ix_j}{R}\frac{d}{dR}(\log h)+\delta_{ij}\right).
\end{equation}
Then,
\[
g=2{\mbox Hess}(u)=2h\left(\frac{1}{R}\frac{d}{dR}(\log h)A+I\right),
\]
where $I$ is the $n\times n$ diagonal matrix and, 
\[
A=\begin{pmatrix} x_1^2 & x_1x_2 & \ldots & x_1x_n \\x_1x_2 & x_2^2 & \ldots & x_2x_n\\ \vdots & \vdots & \ddots & \vdots\\ x_1x_n & x_2x_n & \ldots & x_n^2\end{pmatrix}.
\]
We now have,
\[
\frac{R}{h\frac{d}{dR}(\log h)}{\mbox Hess}(u)=A+\frac{R}{\frac{d}{dR}(\log h)}I,
\]
which yields,
\[
{\mbox det}\left(\frac{R}{h\frac{d}{dR}(\log h)}{\mbox Hess}(u)\right)={\mbox det}\left(A+\frac{R}{\frac{d}{dR}(\log h)}I\right)=\prod_{i=1}^n\left(\frac{R}{\frac{d}{dR}(\log h)}-\lambda_i\right),
\]
where $\lambda_i$ are the eigenvalues of $-A$, which are $\lambda_1=\ldots=\lambda_{n-1}=0$ and $\lambda=-R^2$. Thus, 
\[
{\mbox det}{\mbox Hess}(u)=h^n+Rh^{n-1}h'=h^n+\frac{R}{n}(h^n)',
\]
where $h'$ denotes differentiation with respect to $R$. In terms of the field intensity, we have
\[
{\mbox det}{\mbox Hess}(u)=\frac{H^{n-1}H'}{R^{n-1}}.
\]
The nondegenerecy of $g=f^{\ast}G$ implies that the function $H$ monotone. 

\vspace{0.1in}

1. Using Theorem \ref{t:oneofthemainthe}, the immersion $f(p)=(p,\nabla u(p))$ is minimal if and only if ${\mbox det}{\mbox Hess}(u)=c_0$, where $c_0$ is a positive constant. This is equivalent to solving the following differential equation:
\[
h^n+\frac{R}{n}(h^n)'=c_0.
\]
Solving this ODE, we get,
\begin{equation}\label{e:thefuncth}
h(R)=(c_0+c_1R^{-n})^{1/n},
\end{equation}
where $c_1$ is a real constant and thus the field intensity $H$ is 
\[
H(R)=Rh(R)=(c_0R^{n}+c_1)^{1/n}.
\]
Suppose now that $c_1=0$. Then $h(R)=c_0^{1/n}$ and using (\ref{e:thefunctionh}) we have that $u_{x_ix_j}=c_0^{1/n}\delta_{ij}$. Hence, one can see easily that
\[
u(x_1,\ldots,x_n)=\frac{c_0^{1/n}}{2}(x_1^2+\ldots+x_n^2)+d,
\]
where $d$ is a real constant. Using Theorem \ref{t:oneofthemainthe}, it follows that the graph $f$ is totally geodesic.
If, conversely, $f$ is totally geodesic, the function $u$ must satisfy (\ref{e:totallycondition}). Since $\nabla u=h(R)(x_1,\ldots, x_n)$, we have
\[
u_{x_k}=h(R)x_k.
\]
On the other hand, taking the derivative of (\ref{e:totallycondition}) with respect to $x_k$ we have
\[
2\sum_{i=1}^n a_{ik}x_i+b_k=h(R)x_k,
\]
which yields,
\[
(2a_{kk}-h)x_k+2\sum_{i\neq k}^n a_{ik}x_i+b_k=0.
\]
Therefore, $b_k=0$ and $a_{ik}=0$ for any $i\neq k$. Thus, $h(R)=a_{kk}$ for any $k$, which implies that $a_{kk}=a$ for some constant $a$. Then $h(R)=a$ and it can be obtained by using (\ref{e:thefuncth}) by setting $c_1=0$. 

2. Suppose that the immersion $f$ is Hamiltonian minimal and set $\Phi={\mbox det}{\mbox Hess}(u)$. We have seen from part 1, that
\[
\Phi=\frac{H^{n-1}H'}{R^{n-1}}.
\]
Note that $f$ is minimal if and only $\Phi$ is positive constant. Therefore, we assume that $\Phi'\neq 0$ except to some isolated points. 
Using Theorem \ref{t:oneofthemainthe}, we know that $f$ is Hamiltonian minimal when
\[
\Delta\log\Phi=0,
\]
where $\Delta$ stands for the Laplacian with respect to $g$. This implies,
\begin{equation}\label{e:hminimalcondition}
g(\tilde\nabla\Phi,\tilde\nabla\Phi)=\Phi\Delta\Phi,
\end{equation}
where $\tilde\nabla$ denotes the gradient with respect to $g$. One can show easily that the coefficients $g^{ij}$ of the inverse of $g$ is given by
\[
g^{ij}=-\frac{x_ix_j}{2RH'}\frac{d}{dR}\log\left(\frac{H}{R}\right),\quad \mbox{if}\;\; i\neq j \]

\begin{equation}\label{e:inverseinducedmetric}
g^{ii}=\frac{1}{2H'}\left(1+\frac{R^2-x_i^2}{R}\frac{d}{dR}\log\left(\frac{H}{R}\right)\right).\end{equation}

Using (\ref{e:inverseinducedmetric}), a brief computation gives
\[
g(\tilde\nabla\Phi,\tilde\nabla\Phi)=\frac{(\Phi')^2}{2H'},
\]
and 
\[
\Phi\Delta\Phi=-(n-1)\frac{\Phi \Phi'}{2H'}\frac{d}{dR}\log\left(\frac{H}{R}\right)+R\Phi\frac{d}{dR}\left(\frac{\Phi'}{2RH'}\right)+\frac{\Phi\Phi'}{2RH'}+\frac{(\Phi')^2}{2H'}.
\]
The Hamiltonian minimal condition (\ref{e:hminimalcondition}) now becomes
\[
-(n-1)\frac{\Phi'}{2H'}\frac{d}{dR}\log\left(\frac{H}{R}\right)+R\frac{d}{dR}\left(\frac{\Phi'}{2RH'}\right)+\frac{\Phi'}{2RH'}=0,
\]
and away from possible minimal points (i.e. $\Phi'=0$), we get
\[
\frac{\Phi''}{\Phi'}=(n-1)\frac{d}{dR}\log\left(\frac{H}{R}\right)+\frac{d}{dR}\log H'.
\]
This implies
\[
\Phi'=c_0\Phi,
\]
where $c_0$ is a nonzero real constant. Then 
\[
\frac{H^{n-1}H'}{R^{n-1}}=ke^{c_0R},
\]
and solving the differential equation we obtain the Hamiltonian minimal condition of the Proposition.
\end{proof}

\section{Special isometric embeddings}\label{s5}

It is well known that the space ${\mathbb L}({\mathbb R}^3)$ of oriented lines in the Euclidean 3-space ${\mathbb R}^3$ is identified with the $T{\mathbb S}^2$, where ${\mathbb S}^2$ denotes the round 2-sphere. Consider the K\"ahler metric ${\mathbb G}$ of ${\mathbb L}({\mathbb R}^3)$ derived from the standard K\"ahler structure endowed on ${\mathbb S}^2$. We then prove:

\begin{proof}[{\bf Proof of Theorem \ref{t:minimalembed}}] Consider the round 2-sphere ${\mathbb S}^2$ and let $f:T{\mathbb S}^2\rightarrow T{\mathbb R}^3:(p,V)\mapsto (p,-p\times V)$ be the embedding, where $\times$ is the cross product in ${\mathbb R}^3$. For $X\in T_{(p,V)}T{\mathbb S}^2$, the derivative $df(X)$ is given by
\begin{equation}\label{e:derivative}
\Pi df(X)=\Pi X,\qquad Kdf(X)=-\Pi X\times V-p\times KX.
\end{equation}
The metric ${\mathbb G}$ in $T{\mathbb S}^2$ is
\[
{\mathbb G}_{(p,V)}(X,Y)=g(KX,p\times \Pi Y)-g(\Pi X,p\times KY).
\]
For $X,Y\in T_{(p,V)}T{\mathbb S}^2$, we have
\begin{eqnarray}
(f^{\ast}G)_{(p,V)}(X,Y)&=&G_{f(p,V)}(f_{\ast}X,f_{\ast} Y)\nonumber \\
&=& g(\Pi df X, Kdf Y)+ g(\Pi df Y, Kdf X)\nonumber \\
&=& g(\Pi X, -\Pi Y\times V-p\times KY)+ g(\Pi Y,-\Pi X\times V-p\times KX)\nonumber\\
&=& -g(\Pi X,p\times KY)-g(\Pi Y,p\times KX),\nonumber
\end{eqnarray}
which shows that $f^{\ast}G={\mathbb G}$ and thus $f$ is an isometric embedding.

We now show that $f$ is minimal. Denote by $\nabla,\overline\nabla$ the Levi-Civita connections of $({\mathbb R}^3,\left<.,.\right>)$ and $(T{\mathbb R}^3,G)$, respectively and also denote respectively by $D,\overline D$ the Levi-Civita connections of $({\mathbb S}^2,g)$ and $(T{\mathbb S}^2,{\mathbb G})$.

A brief computation gives
\[
\overline\nabla_{dfX}dfY=(\nabla_{\Pi X}\Pi Y,-\nabla_{\Pi X}\Pi Y\times V-p\times\nabla_{\Pi X}KY-\Pi Y\times KX-\Pi X\times KY+\left<\Pi X,V\right>p\times \Pi Y)
\]
and using the fact that $\overline D_XY=(D_{\Pi X}\Pi Y,D_{\Pi X}KY-\left<V,\Pi X\right>\Pi Y)$, we have
\[
df(\overline D_XY)=(\nabla_{\Pi X}\Pi Y-\left<\Pi X,\Pi Y\right>p,-\nabla_{\Pi X}\Pi Y\times V+\left<\Pi X,\Pi Y\right>p\times V
\]
\[
\qquad\qquad\qquad\qquad\qquad\qquad\qquad\qquad\qquad\qquad-p\times \nabla_{\Pi X}KY+\left<V,\Pi X\right>p\times \Pi Y)
\]
The second fundamental form $h$ of $f$ is given by
\begin{eqnarray}
h(dfX,dfY)&=&\overline\nabla_{dfX}dfY-df(\overline D_XY)\nonumber\\
&=&(\left<\Pi X,\Pi Y\right>p,\left<\Pi X,\Pi Y\right>p\times V-\Pi Y\times KX-\Pi X\times KY)\nonumber
\end{eqnarray}
Suppose that $(p,V)\in T{\mathbb S}^2$ and $|V|\neq 0$. Consider the following orthogonal basis of $T_{(p,V)}T{\mathbb S}^2$:
\[
E_1=(V,p\times V),\qquad E_2=(p\times V,V),\qquad E_3=(V,-p\times V),\qquad E_4=(-p\times V,V).
\]
Thus,
\[
|E_1|^2=-|E_2|^2=-|E_3|^2=|E_4|^2=2|V|^2.
\]
Using (\ref{e:derivative}) we have
\[
df E_1=(V,V),\qquad df E_2=(p\times V,|V|^2p-p\times V),\qquad 
df E_3=(V,-V),
\]
\[
df E_4=(-p\times V,-|V|^2p-p\times V).
\]
The second fundamental form for this basis is
\[
h(df E_1,df E_1)=h(df E_3,df E_3)=(|V|^2p,-|V|^2p\times V),
\]
and
\[
h(df E_2,df E_2)=h(df E_4,df E_4)=(|V|^2p,-|V|^2p\times V-2V),
\]
The mean curvature $H$ of $f$ is
\[
H=\frac{1}{2|V|^2}\Big(h(df E_1,df E_1)-h(df E_2,df E_2)-h(df E_3,df E_3)+h(df E_4,df E_4)\Big),
\]
which clearly vanishes and therefore the embedding $f$ is minimal.

Suppose that $V=0$ and let $\xi$ be a unit vector in $T_{p}{\mathbb S}^2$. Then we obtain the following orthonormal frame for $T_{(p,0)}T{\mathbb S}^2$:
\[
E_1=(\xi,p\times\xi),\qquad E_2=(p\times\xi,\xi),\qquad E_3=(\xi,-p\times\xi),\qquad E_4=(-p\times\xi,\xi).
\]
We then have
\[
df E_1=(\xi,\xi),\qquad df E_2=(p\times\xi,-p\times\xi),\qquad 
df E_3=(\xi,-\xi),
\]
\[
df E_4=(-p\times\xi,-p\times\xi).
\]
The second fundamental form for this basis is
\[
h(df E_1,df E_1)=h(df E_3,df E_3)=
h(df E_2,df E_2)=h(df E_4,df E_4)=(p,0),
\]
and thus,
\[
H=h(df E_1,df E_1)-h(df E_2,df E_2)-h(df E_3,df E_3)+h(df E_4,df E_4)=0,
\]
and the Theorem follows.
\end{proof}

Note that $(T{\mathbb S}^2,{\mathbb G})$ can't be isometrically embedded in $(T{\mathbb R}^3,G_0)$, where $G_0$ is the Sasakian metric.
\vspace{0.1in}




\end{document}